\documentclass[12pt,reqno]{amsart}
\usepackage[colorlinks=true, pdfstartview=FitV, linkcolor=blue, citecolor=blue, urlcolor=blue]{hyperref}
\usepackage{amssymb,amsmath, amscd}
\usepackage{ragged2e}
\usepackage{times, verbatim}
\usepackage{mathdots}
\usepackage{array}
\usepackage{graphicx}
\usepackage[english]{babel}
 \usepackage[usenames, dvipsnames]{color}
\usepackage{amsmath,amssymb,amsfonts}
\usepackage{enumerate, enumitem}
\usepackage{MnSymbol}
\usepackage{anysize}
\usepackage{enumitem}
\usepackage{bigints}
\usepackage{braket}
\usepackage[english]{babel}
\usepackage{blindtext}
\usepackage{mathtools}
\usepackage{graphicx}
\usepackage{tikz-cd}
\usepackage[a4paper,verbose]{geometry}
\marginsize{2.5cm}{2.5cm}{2.5cm}{2.5cm}
\input xy
\xyoption{all}
\usepackage{pb-diagram}
\usepackage[all]{xy}
\input xy
\xyoption{all}
\usepackage{scalerel}

\DeclareFontFamily{OT1}{rsfs}{}
\DeclareFontShape{OT1}{rsfs}{n}{it}{<-> rsfs10}{}
\DeclareMathAlphabet{\mathscr}{OT1}{rsfs}{n}{it}
\DeclareRobustCommand\longtwoheadrightarrow
     {\relbar\joinrel\twoheadrightarrow}

\numberwithin{equation}{section}
\numberwithin{equation}{subsection}
\renewcommand*{\theequation}{%
  \ifnum\value{subsection}=0 %
    \thesection
  \else
    \thesubsection
  \fi
  .\arabic{equation}%
}

\newtheorem{thm}[equation]{Theorem}
\newtheorem{prop}[equation]{Proposition}
\newtheorem{corollary}[equation]{Corollary}

\newtheorem{lemma}[equation]{Lemma}
\newtheorem{defn}[equation]{Definition}

\theoremstyle{definition}
\newtheorem*{funding}{Funding}

\theoremstyle{remark}

\begin{document}
\theoremstyle{plain}

\makeatletter
\DeclareRobustCommand{\bigboxplus@}{%
  \mathop{\vphantom{\sum@}\scaleobj{0.8}{\scalerel*{\boxplus}{\sum}}}%
}
\newcommand{\bigboxplus}{\DOTSB\bigboxplus@\slimits@}
\DeclareRobustCommand{\bigboxtimes@}{%
  \mathop{\vphantom{\sum@}\scaleobj{0.8}{\scalerel*{\boxtimes}{\sum}}}%
}
\newcommand{\bigboxtimes}{\DOTSB\bigboxtimes@\slimits@}
\makeatother

\newcommand{\Hecke}{\mathcal{H}}
\newcommand{\Liea}{\mathfrak{a}}
\newcommand{\Cmg}{C_{\mathrm{mg}}}
\newcommand{\Cinftyumg}{C^{\infty}_{\mathrm{umg}}}
\newcommand{\Cfd}{C_{\mathrm{fd}}}
\newcommand{\Cinftyfd}{C^{\infty}_{\mathrm{ufd}}}
\newcommand{\sspace}{\Gamma \backslash G}
\newcommand{\PP}{\mathcal{P}}
\newcommand{\bfP}{\mathbf{P}}
\newcommand{\bfQ}{\mathbf{Q}}
\newcommand{\Siegel}{\mathfrak{S}}
\newcommand{\g}{\mathfrak{g}}
\newcommand{\A}{\mathbb{A}}
\newcommand{\Q}{\mathbb{Q}}
\newcommand{\Gm}{\mathbb{G}_m}
\newcommand{\Nm}{\mathbb{N}m}
\newcommand{\ii}{\mathfrak{i}}
\newcommand{\II}{\mathfrak{I}}

\newcommand{\kk}{\mathfrak{k}}
\newcommand{\nn}{\mathfrak{n}}
\newcommand{\tF}{\widetilde{F}}
\newcommand{\p}{\mathfrak{p}}
\newcommand{\m}{\mathfrak{m}}
\newcommand{\bb}{\mathfrak{b}}
\newcommand{\Ad}{{\rm Ad}\,}
\newcommand{\ttt}{\mathfrak{t}}
\newcommand{\frakt}{\mathfrak{t}}
\newcommand{\U}{\mathcal{U}}
\newcommand{\Z}{\mathbb{Z}}
\newcommand{\bfG}{\mathbf{G}}
\newcommand{\bfT}{\mathbf{T}}
\newcommand{\R}{\mathbb{R}}
\newcommand{\ST}{\mathbb{S}}
\newcommand{\h}{\mathfrak{h}}
\newcommand{\bC}{\mathbb{C}}
\newcommand{\C}{\mathbb{C}}
\newcommand{\N}{\mathbb{N}}
\newcommand{\qH}{\mathbb {H}}
\newcommand{\temp}{{\rm temp}}
\newcommand{\Hom}{{\rm Hom}}
\newcommand{\Aut}{{\rm Aut}}
\newcommand{\rk}{{\rm rk}}
\newcommand{\Ext}{{\rm Ext}}
\newcommand{\End}{{\rm End}\,}
\newcommand{\Ind}{{\rm Ind}}
\newcommand{\ind}{{\rm ind}}
\newcommand{\Irr}{{\rm Irr}}
\def\circG{{\,^\circ G}}
\def\M{{\rm M}}
\def\diag{{\rm diag}}
\def\Ad{{\rm Ad}}
\def\As{{\rm As}}
\def\wG{{\widehat G}}
\def\G{{\rm G}}
\def\SL{{\rm SL}}
\def\PSL{{\rm PSL}}
\def\GSp{{\rm GSp}}
\def\PGSp{{\rm PGSp}}
\def\Sp{{\rm Sp}}
\def\St{{\rm St}}
\def\GU{{\rm GU}}
\def\SU{{\rm SU}}
\def\U{{\rm U}}
\def\GO{{\rm GO}}
\def\GL{{\rm GL}}
\def\PGL{{\rm PGL}}
\def\GSO{{\rm GSO}}
\def\GSpin{{\rm GSpin}}
\def\GSp{{\rm GSp}}

\newcommand{\Ssp}{\mathcal S}
\newcommand{\W}{\mathcal W}

\newcommand{\Sym}{\operatorname{Sym}}
\newcommand{\wh}{\widehat\otimes}
\newcommand{\detabs}{|\det|}
\newcommand{\jet}{\operatorname{jet}}
\newcommand{\RS}{\operatorname{RS}}

\def\Gal{{\rm Gal}}
\def\SO{{\rm SO}}
\def\O{{\rm  O}}
\def\sym{{\rm sym}}
\def\St{{\rm St}}
\def\Sp{{\rm Sp}}
\def\tr{{\rm tr\,}}
\def\ad{{\rm ad\, }}
\def\Ad{{\rm Ad\, }}
\def\rank{{\rm rank\,}}

\def\Ext{{\rm Ext}}
\def\Hom{{\rm Hom}}
\def\Alg{{\rm Alg}}
\def\GL{{\rm GL}}
\def\SO{{\rm SO}}
\def\G{{\rm G}}
\def\U{{\rm U}}
\def\St{{\rm St}}
\def\Wh{{\rm Wh}}
\def\RS{{\rm RS}}
\def\ind{{\rm ind}}
\def\Ind{{\rm Ind}}
\def\csupp{{\rm csupp}}

\title[Exceptional poles of archimedean Asai $L$-functions]
{Exceptional poles of archimedean Asai $L$-functions}
\author[Akash Yadav]{Akash Yadav}
\address{(A. Yadav) Human-Centered Artificial Intelligence Research Institute, Ewha Womans University, Seoul 03760, Republic of Korea}
\email{akaschampion@ewha.ac.kr, akaschampion@gmail.com}

\subjclass[2020]{Primary 11F70; Secondary 20G20, 22E45}
\keywords{Archimedean Asai $L$-functions,  Exceptional Poles}

\begin{abstract}
Let $\pi$ be an irreducible generic representation of
$\GL_n(\mathbb{C})$. We classify the exceptional poles of the associated
Asai $L$-function $L(s,\pi,\mathrm{As})$ and give a representation-theoretic
characterization of them. When $\pi$ is in general position, we further
express $L(s,\pi,\mathrm{As})$ in terms of the exceptional $L$-factors
attached to the irreducible constituents occurring in the derivatives
of $\pi$.
\end{abstract}
\maketitle

\section{Introduction}\label{s1}

Cogdell and Piatetski-Shapiro introduced the notion of an exceptional
pole in local Rankin--Selberg theory \cite{CPS2017}. Such poles connect
the analytic behavior of local zeta integrals with properties of the
representations, including distinction. This connection comes from
the equivariant functionals given by the leading Laurent coefficients
at these poles.

The non-archimedean theory of exceptional poles has been studied extensively; see, for
example, \cite{CPS2017,Mat2015,Jo2020}. By contrast, the archimedean
theory remains considerably less developed. Recently, the classification of exceptional poles of
Rankin--Selberg integrals over $\R$ was completed
in \cite{JNY2026}. The purpose of this paper is to study exceptional poles of the Flicker
integrals \cite{Fli1993} representing the Asai \(L\)-function \(L(s,\pi,\mathrm{As})\),
where \(\pi\) is an irreducible generic representation of
\(\operatorname{GL}_n(\C)\). 

Fix $\pi$ as above. Let $\psi\colon\C\to \mathbb{S}^1$ be a nontrivial unitary additive character that is trivial on $\R$, and denote by $\mathcal W(\pi,\psi)$ the corresponding Whittaker model.
Write $\mathcal S$ for the Schwartz space on $\R^n$, let $e_n=(0,\dots,0,1)\in\R^n$, and let $N_n$ denote
the standard upper-triangular unipotent subgroup of $\operatorname{GL}_n$.  We consider the family of Flicker integrals $\mathcal{J}(\pi) = \{I(s,W,\Phi) \mid W \in \mathcal{W}(\pi, \psi), \Phi \in \mathcal{S}\}$, where
\[
I(s,W,\Phi) = \int_{N_n(\R) \backslash \operatorname{GL}_n(\R)} W(g)\Phi(e_n g)|\det(g)|_\R^s \, dg.
\]
By \cite[Theorem~1]{BP2021}, these integrals converge absolutely for
$\operatorname{Re}(s)$ sufficiently large and admit meromorphic
continuation to $\mathbb C$. Moreover, the normalized integrals
\[
\Lambda_s(W,\Phi)
=
\frac{I(s,W,\Phi)}{L(s,\pi,\mathrm{As})}
\]
are entire functions of $s$. Since $\Lambda_{s_0}\neq0$ for every
$s_0\in\C$ \cite[Theorem~6.1]{JY2026}, the poles of
$\mathcal J(\pi)$, with their maximal orders, coincide
with those of $L(s,\pi,\mathrm{As})$.

We have a natural filtration $\mathcal{S} = \mathcal{S}^0 \supset \mathcal{S}^1 \supset \cdots \supset \mathcal{S}^m \supset \cdots$, with $\mathcal{S}^m$ defined as the set of functions in $\mathcal{S}$ whose partial derivatives of total order less than $m$
vanish at $0$. Suppose that $s_0$ is a pole of
$\mathcal J(\pi)$, and let $d$ be its maximal order among the
integrals in this family. Then, for every $W$ and $\Phi$, we may write
\[
I(s,W,\Phi)
=
\frac{B_{s_0}(W,\Phi)}{(s-s_0)^d}
+\cdots ,
\]
where $B_{s_0}$ is a nonzero bilinear form on
$\mathcal W(\pi,\psi)\times\mathcal S$. We say that $s_0$ is an exceptional pole of level $m$ for $\mathcal{J}(\pi)$ if $B_{s_0}$ vanishes identically on $\mathcal W(\pi,\psi)\times\mathcal S^{m+1}$ but not on
$\mathcal W(\pi,\psi)\times\mathcal S^m$.

Let \(\operatorname{Sym}^m(\mathbb C^n)\) denote the \(m\)-th symmetric
power of the standard representation of
\(\operatorname{GL}_n(\mathbb C)\). We write
\(\widehat{\otimes}\) for the completed projective tensor product.
Our first main result characterizes exceptional poles. 

\begin{thm}\label{1.1}
Let $\pi$ be an irreducible generic representation of
$\operatorname{GL}_n(\C)$, and let $m\in\mathbb Z_{\geq0}$
and $s_0\in\C$. The following statements are equivalent:
\begin{enumerate}[label=$(\roman*)$]
    \item
    $s_0$ is an exceptional pole of level $m$ for
    $\mathcal J(\pi)$.
    \item
    $\operatorname{Hom}_{\operatorname{GL}_n(\R)}
    \left(
    \pi\widehat{\otimes}\operatorname{Sym}^m(\C^n),
    \lvert\det\rvert_\mathbb R^{-s_0}
    \right)\neq0$.
\end{enumerate}
\end{thm}

For $m=0$, Theorem~\ref{1.1} recovers the
$\operatorname{GL}_n(\R)$-distinction criterion of
\cite[Theorem~6.7]{JY2026}, applied to
$\pi\otimes|\det|_{\C}^{s_0/2}$.
    
The implication $(i)\Rightarrow(ii)$ follows from the
equivariance and filtration properties established at the end
of Section~\ref{s2}. The main difficulty is the converse. We first study a filtration of
$\pi\widehat{\otimes}\operatorname{Sym}^m(\mathbb C^n)$
whose successive quotients are induced from algebraic twists of the inducing characters of $\pi$. An orbit analysis of equivariant functionals on these quotients shows that the nonvanishing of the Hom space in $(ii)$ implies that $s_0$ is a pole of $L(s,\pi,\mathrm{As})$.

Composing a nonzero functional in $(ii)$ with the homogeneous degree-$m$ Taylor map in the Schwartz variable gives an equivariant bilinear form that vanishes on
\(\mathcal W(\pi,\psi)\times\mathcal S^{m+1}\) but not on
\(\mathcal W(\pi,\psi)\times\mathcal S^m\).
The uniqueness theorem of Chen,
Sun, and Weng
\cite[Theorem~1.4 and Remark~3.3]{CSW2026}
identifies this form with $\Lambda_{s_0}$ up to a nonzero
scalar. Since $B_{s_0}$ is also a nonzero scalar multiple
of $\Lambda_{s_0}$, the pole is exceptional of level $m$.

As an application of Theorem~\ref{1.1}, we express the Asai $L$-function of a representation in general position in terms of the exceptional Asai $L$-factors attached to the irreducible constituents of its derivatives. We refer the reader to Section~\ref{s5} for the precise definitions of these derivatives, the exceptional factors $L_{\mathrm{ex}}$, and the $\operatorname{l.c.m.}$ notation.

\begin{thm}\label{1.2}
Let $\pi$ be an irreducible generic representation of
$\operatorname{GL}_n(\mathbb C)$ in general position, as defined in
Section~\ref{s2}. For each \(k\) with \(0 \le k < n\), write
\[
\pi^{(k)} = \bigoplus_i \pi_{i}^{(k)}, 
\]
where the \(\pi_{i}^{(k)}\) denote the irreducible constituents of the $k$-th derivative of \(\pi\). Then
\[
L(s,\pi,\mathrm{As})^{-1}
=
\underset{\substack{0\leq k<n\\ i}}{\operatorname{l.c.m.}}
\left\{
L_{\mathrm{ex}}
\left(s,\pi_i^{(k)},\mathrm{As}\right)^{-1}
\right\}.
\]
\end{thm}

The key step in the proof is to show that every pole of $L(s,\pi,\mathrm{As})$ occurs, with the same order, as an
exceptional pole of an irreducible constituent of a
derivative of $\pi$; see
Theorem~\ref{order-characterization}.
The Rankin--Selberg analogue of Theorem~\ref{1.2} is
proved in \cite[Theorem~1.2]{JNY2026}.

The paper is organized as follows. In Section~\ref{s2}, we fix
notation and recall the necessary definitions. In Section~\ref{s3},
we study the Hom space in Theorem~\ref{1.1}~$(ii)$ and show that its
nonvanishing implies that $s_0$ is a pole of the Asai \(L\)-function.
In Section~\ref{s4}, we prove Theorem~\ref{1.1}. In Section~\ref{s5}, we define the exceptional $L$-factors, compute the
derivatives of $\pi$ explicitly, and prove Theorem~\ref{1.2}.

\begin{section}{Notation and Preliminaries}\label{s2}

Let \(\mathbb{R}\) and \(\mathbb{C}\) denote the fields of real and complex numbers, equipped with their normalized absolute values \(\lvert\>\cdot\>\rvert_{\mathbb{R}}\) and \(\lvert\>\cdot\>\rvert_{\mathbb{C}}\), respectively. For \(n \geq 1\), let \(G_n\) denote the algebraic group \(\GL_n\), and let \(B_n\), \(A_n\), and \(N_n\) denote its subgroups of upper triangular, diagonal, and unipotent upper triangular matrices, respectively. We write \(e_n = (0, \dots, 0, 1) \in \mathbb{R}^n\), viewed as a row vector. Let \(\mathbb{S}^1\) denote the multiplicative group of complex numbers whose absolute value is \(1\).

By a representation of \(G_n(\C)\), we mean a Casselman--Wallach representation, that is, a smooth admissible Fr\'echet representation of moderate growth and of finite length. We write \(\Irr(G_n(\C))\) for the set of isomorphism classes of irreducible representations of \(G_n(\C)\). By a slight abuse of notation, we identify a representation $\pi \in \operatorname{Irr}(G_n(\mathbb{C}))$ with its underlying vector space. We use \(\widehat{\otimes}\) to denote the completed projective tensor
product. All \(\operatorname{Hom}\) spaces of equivariant linear maps are assumed to consist of continuous linear maps.

Let $P=P_{n_1,\ldots,n_k}$ be the standard parabolic subgroup of $G_n$
associated to the ordered partition $n=n_1+\cdots+n_k$, and write $P=MU$ for its Levi
decomposition, where
\[
M = G_{n_1}\times \cdots \times G_{n_k}.
\]
We write $\mathfrak g_n = \operatorname{Lie}(G_n(\C))\otimes_{\R}\C$
for the complexified Lie algebra of $G_n(\C)$.

For each $1\le i\le k$, let $\tau_i\in \Irr(G_{n_i}(\C))$, and set
$\sigma=\tau_1\boxtimes\cdots\boxtimes\tau_k$, viewed as a representation of
$M(\C)$. We denote by
\[
\tau_1 \times \cdots \times \tau_k
\]
the normalized parabolic induction
$\operatorname{Ind}_{P(\C)}^{G_n(\C)}(\sigma)$.

Fix a nontrivial additive character
\(\psi:\mathbb C\to\mathbb S^1\) which is trivial on \(\mathbb R\).
For each \(n\geq 1\), we define the generic character \(\psi_n\) of
\(N_n(\mathbb C)\) by
\[
\psi_n(u)
=
\psi\!\left((-1)^n\sum_{i=1}^{n-1}u_{i,i+1}\right),
\qquad u\in N_n(\mathbb C).
\]
In particular, \(\psi_n\) is trivial on \(N_n(\mathbb R)\).

A representation \(\pi\in\Irr(G_n(\mathbb C))\) is \emph{generic} if
there is a nonzero continuous linear functional
\(\lambda:\pi\to\mathbb C\) such that
\[
\lambda(\pi(u)v)=\psi_n(u)\lambda(v)
\qquad
(u\in N_n(\mathbb C),\ v\in\pi).
\]
We write $\Irr_{\mathrm{gen}}(G_n(\mathbb C))$ for the subset
of generic representations. For
$\pi\in\Irr_{\mathrm{gen}}(G_n(\mathbb C))$, let
$\mathcal W(\pi,\psi)$ denote its Whittaker model with respect
to $\psi_n$. We endow $\mathcal{W}(\pi,\psi)$ with the Fréchet topology induced by $\pi$.

By \cite{SV1980} (see also \cite[Appendix~A]{Kem2015}), every representation
\(\pi\in\Irr_{\mathrm{gen}}(G_n(\C))\) is isomorphic to an irreducible
principal series representation of the form
\[
\pi=\chi_1\times\cdots\times\chi_n,
\]
where each \(\chi_i\) is a character of \(\C^\times\). For \(z\in\mathbb C^\times\), set
\[
[z]=\frac{z}{|z|_{\mathbb C}^{1/2}}.
\]
For \(m\in\mathbb Z\) and \(s\in\mathbb C\), define the character
\(\chi_{m,s}\) of \(\mathbb C^\times\) by
\[
\chi_{m,s}(z)=[z]^m|z|_{\mathbb C}^{s}.
\]
Every smooth character of \(\mathbb C^\times\) is uniquely of this
form. 

We shall mainly consider principal series representations of the form
\[
\pi = \chi_{m_1,s_1} \times \dots \times \chi_{m_n,s_n},
\qquad
m_i \in \mathbb{Z}, \quad s_i \in \mathbb{C}.
\]
For each \(1 \leq i \leq n\), we set
\[
a_i = s_i + \frac{m_i}{2},
\qquad
b_i = s_i - \frac{m_i}{2}.
\]
Following \cite{CC1999}, we say that \(\pi\) is in \emph{general position} if
\[
\begin{aligned}
a_i, \, b_i &\notin \mathbb{Z} && (1 \leq i \leq n), \\
a_i - a_j, \, b_i - b_j &\notin \mathbb{Z} && (1 \leq i \neq j \leq n).
\end{aligned}
\]
Under these hypotheses, the representation \(\pi\) is irreducible and generic.

Set
\[
\Gamma_{\mathbb R}(s)
=
\pi^{-s/2}\Gamma\left(\frac{s}{2}\right)
\qquad\text{and}\qquad
\Gamma_{\mathbb C}(s)
=
2(2\pi)^{-s}\Gamma(s),
\]
where \(\Gamma(s)\) denotes the usual Gamma function.

Let $W_{\mathbb C}=\mathbb C^\times
$ be the Weil group of \(\mathbb C\), and let
\[
W_{\mathbb R}
=
\mathbb C^\times\sqcup j\mathbb C^\times
\]
be the Weil group of \(\mathbb R\), where
$j^2=-1$ and $jwj^{-1}=\overline{w}$ for every $w\in\mathbb C^\times$.
An admissible representation of \(W_{\mathbb C}\) or
\(W_{\mathbb R}\) is a continuous semisimple representation on a
finite-dimensional complex vector space. For such a representation
\(\sigma\), we denote by \(L(s,\sigma)\) the local \(L\)-factor
associated with \(\sigma\) as in~\cite{Tate1979}.

Let  $\phi:W_{\mathbb C}\longrightarrow\operatorname{GL}(V)
$ be an admissible representation. Its \emph{Asai representation}
\[
\operatorname{As}(\phi):
W_{\mathbb R}\longrightarrow\operatorname{GL}(V\otimes V)
\]
is defined on pure tensors by
\[
\begin{aligned}
\As(\phi)(z)(u\otimes v)
&=\phi(z)u\otimes\phi(jzj^{-1})v,
&& z\in W_{\mathbb{C}},\\
\As(\phi)(j)(u\otimes v)
&=v\otimes\phi(j^2)u,
\end{aligned}
\qquad u,v\in V.
\]
Let \(\pi\in\Irr(G_n(\mathbb C))\), and let
\[
\phi_\pi:
W_{\mathbb C}\longrightarrow\operatorname{GL}_n(\mathbb C)
\]
be the admissible representation of $W_{\mathbb C}$ associated with \(\pi\) by the local
Langlands correspondence \cite{Kna1994}. We define
\[
L(s,\pi,\operatorname{As})
=
L\bigl(s,\operatorname{As}(\phi_\pi)\bigr),
\]
and call it the \emph{Asai \(L\)-function} of \(\pi\).

We now record the equivariance properties underlying
Theorem~\ref{1.1}. For analogous constructions in the
Rankin--Selberg setting, see Chai~\cite{Chai2015} and
\cite{JNY2026}.

Fix $\pi\in\Irr_{\mathrm{gen}}(G_n(\mathbb C))$, and recall
the filtration $\{\mathcal S^m\}_{m\geq0}$ of
the Schwartz space $\mathcal S$ of $\R^n$ introduced in Section~\ref{s1}.
The action of $G_n(\mathbb R)$ on $\mathcal S$ given by
\begin{equation}\label{eq:schwartz-row-action-asai}
(g\cdot\Phi)(x)=\Phi(xg),
\qquad
g\in G_n(\mathbb R),\quad x\in\mathbb R^n,
\end{equation}
preserves this filtration. For every $m\geq0$, taking the
homogeneous degree-$m$ Taylor term at the origin induces
a $G_n(\mathbb R)$-equivariant isomorphism
\begin{equation}\label{eq:schwartz-graded-asai}
\mathcal S^m/\mathcal S^{m+1}
\cong
\operatorname{Sym}^m(\mathbb C^n).
\end{equation}
Here, $\operatorname{Sym}^m(\mathbb C^n)$ is realized as the
space of homogeneous polynomials of degree $m$ in a row
variable, with action
\[
(g\cdot P)(x)=P(xg).
\]

We let $G_n(\mathbb R)$ act on $\mathcal W(\pi,\psi)$
by right translation:
\[
(g\cdot W)(h)=W(hg),
\qquad
g\in G_n(\mathbb R),\quad h\in G_n(\mathbb C).
\]
A change of variables in the Flicker integral gives
\[
I(s,g\cdot W,g\cdot\Phi)
=
\lvert\det g\rvert_\mathbb R^{-s}I(s,W,\Phi),
\qquad g\in G_n(\mathbb R).
\]
Consequently, at any pole $s_0$ of the family of Flicker integrals, comparison of leading Laurent coefficients gives
\begin{equation}\label{eq:asai-leading-coefficient-equivariance}
B_{s_0}(g\cdot W,g\cdot\Phi)
=
\lvert\det g\rvert_\mathbb R^{-s_0}B_{s_0}(W,\Phi),
\qquad g\in G_n(\mathbb R).
\end{equation}
Moreover, $B_{s_0}$ is continuous
\cite[Theorem~6.1 and Remark~6.2]{JY2026}.

Suppose now that $s_0$ is an exceptional pole of level $m$
for $\pi$. By definition, the restriction of $B_{s_0}$ to
$\mathcal W(\pi,\psi)\times\mathcal S^m$ descends to a
nonzero continuous bilinear form on
\[
\mathcal W(\pi,\psi)\times
\bigl(\mathcal S^m/\mathcal S^{m+1}\bigr).
\]
By \eqref{eq:schwartz-graded-asai} and
\eqref{eq:asai-leading-coefficient-equivariance}, this form
determines a nonzero element of
\[
\operatorname{Hom}_{G_n(\mathbb R)}
\left(
\pi\widehat{\otimes}\operatorname{Sym}^m(\mathbb C^n),
\lvert\det\rvert_\mathbb R^{-s_0}
\right).
\]
This proves the implication from $(i)$ to $(ii)$ in
Theorem~\ref{1.1}. The converse will be proved in
Section~\ref{s4}.

\end{section}

\begin{section}{Equivariant Functionals and Poles of the Asai L-Function}\label{s3}

We first recall the primary decomposition with respect to
infinitesimal characters. The proof of
\cite[Lemma~3.1]{JNY2026} applies equally to $G_n(\C)$.
Write $Z(\mathfrak g_n)$ for the center of the universal
enveloping algebra of $\mathfrak g_n$.

\begin{lemma}\label{lem:primary-decomposition}
Let \(X\) be a representation of
\(G_n(\C)\). For an infinitesimal character \(\zeta\), let
\(\mathfrak m_\zeta\) denote the corresponding maximal ideal of
\(Z(\mathfrak g_n)\), and set
\[
X[\zeta]
=
\left\{
x\in X:
\mathfrak m_\zeta^N\cdot x=0
\text{ for some }N\geq1
\right\}.
\]
Then:
\begin{enumerate}
\item $X$ decomposes as a finite topological direct sum
\[
X = \bigoplus_{\zeta} X[\zeta].
\]
\item Each projection $X \to X[\zeta]$ is continuous and $G_n(\C)$-equivariant.
\item The functor
\[
X \longmapsto X[\zeta]
\]
is exact on the category of Casselman--Wallach representations.
\end{enumerate}
\end{lemma}

Next, we study algebraic twists of irreducible generic representations. Let
\[
\pi=
\chi_{m_1,s_1}\times\cdots\times\chi_{m_n,s_n}
\]
be an irreducible generic representation of $G_n(\C)$, where
$m_i\in\mathbb Z$ and $s_i\in\mathbb C$. Fix $m\in\mathbb Z_{\geq0}$, and put
$\rho=\Sym^m(\mathbb C^n)$. Let
\[
B_n=A_nN_n
\]
be the standard Borel subgroup of $G_n$. Write
\[
\sigma=
\chi_{m_1,s_1}\boxtimes\cdots\boxtimes\chi_{m_n,s_n},
\qquad
\pi=\Ind_{B_n(\C)}^{G_n(\C)}(\sigma).
\]

Let $V=\mathbb C^n$ and consider the flag
\[
0=V_0\subset V_1\subset\cdots\subset V_n=V
\]
stabilized by $B_n(\C)$, where
\[
\dim(V_j/V_{j-1})=1,
\qquad 1\leq j\leq n.
\]
Put
\[
W_j=V_j/V_{j-1},
\qquad 1\leq j\leq n.
\]

We now construct a finite $B_n(\C)$-stable filtration of
$\rho=\Sym^m(V)$. Set
\[
r_j=(m+2)^{j-1},
\qquad 1\leq j\leq n,
\]
and equip $V$ with the increasing filtration for which the graded
piece $W_j$ has degree $r_j$. Equip $\Sym^m(V)$ with the induced
multiplicative filtration. Thus, $\mathcal F_d\rho$ is spanned by
monomials $v_1\cdots v_m$ for which there exist indices
$j_1,\ldots,j_m$ such that
\[
v_q\in V_{j_q},
\qquad
r_{j_1}+\cdots+r_{j_m}\leq d.
\]
Each $\mathcal F_d\rho$ is $B_n(\C)$-stable.

Let
\[
\mathcal C_m=
\left\{
\mathbf c=(c_1,\ldots,c_n)\in
\mathbb Z_{\geq0}^{n}
:
c_1+\cdots+c_n=m
\right\},
\]
and define
\[
\operatorname{wt}(\mathbf c)=\sum_{j=1}^{n}c_jr_j.
\]
Since $0\leq c_j\leq m$, the map
$\mathbf c\mapsto\operatorname{wt}(\mathbf c)$ is injective. Order the elements of
$\mathcal C_m$ as
\[
\mathbf c^{(1)},\ldots,\mathbf c^{(N)},
\qquad
\operatorname{wt}(\mathbf c^{(1)})<\cdots<
\operatorname{wt}(\mathbf c^{(N)}),
\]
where
\[
N=\binom{m+n-1}{m},
\]
and define
\[
F^0\rho=0,
\qquad
F^i\rho=\mathcal F_{\operatorname{wt}(\mathbf c^{(i)})}\rho,
\qquad 1\leq i\leq N.
\]
Then
\[
0=F^0\rho\subset F^1\rho\subset\cdots\subset F^N\rho=\rho
\]
is a $B_n(\C)$-stable filtration, and its successive quotients are
\[
F^i\rho/F^{i-1}\rho
\cong
\bigotimes_{u=1}^{n}
\Sym^{c_u^{(i)}}(W_u)
\]
as $B_n(\C)$-representations, with $N_n(\C)$ acting trivially on both sides.

The tensor identity gives
\[
\pi\widehat{\otimes}\rho
\cong
\Ind_{B_n(\C)}^{G_n(\C)}
\bigl(\sigma\otimes\rho|_{B_n(\C)}\bigr).
\]
Here $\Ind$ denotes normalized smooth induction of
finite-dimensional $B_n(\C)$-modules, whose $N_n(\C)$-action need not be trivial.
Tensoring the filtration of $\rho$ with $\sigma$ and applying this functor yields a filtration by closed $G_n(\C)$-stable
subspaces
\[
0=M_0\subseteq M_1\subseteq\cdots\subseteq M_N
=\pi\widehat{\otimes}\rho,
\]
where
\[
M_i=
\Ind_{B_n(\C)}^{G_n(\C)}
\bigl(\sigma\otimes F^i\rho\bigr).
\]
The graded pieces are
\[
\begin{aligned}
M_i/M_{i-1}
&\cong
\Ind_{B_n(\C)}^{G_n(\C)}
\left(
\sigma\otimes
\left(F^i\rho/F^{i-1}\rho\right)
\right)\\
&\cong
\chi_{m_1+c_1^{(i)},\,s_1+\frac{c_1^{(i)}}2}
\times\cdots\times
\chi_{m_n+c_n^{(i)},\,s_n+\frac{c_n^{(i)}}2}.
\end{aligned}
\]
Here we have used
\[
\chi_{m_u,s_u}\otimes\Sym^{c_u^{(i)}}(\mathbb C)
\cong
\chi_{m_u+c_u^{(i)},\,s_u+\frac{c_u^{(i)}}2}.
\]
For $\mathbf c=(c_1,\ldots,c_n)\in\mathcal C_m$, put
\begin{equation}\label{eq:sigma-c}
\sigma_{\mathbf c}
=
\chi_{m_1+c_1,\,s_1+\frac{c_1}{2}}
\times\cdots\times
\chi_{m_n+c_n,\,s_n+\frac{c_n}{2}}.
\end{equation}

Now let $s_0\in\mathbb C$ and assume that
\[
\operatorname{Hom}_{G_n(\R)}
\left(
\pi\widehat{\otimes}\rho,
\lvert\det\rvert_\mathbb R^{-s_0}
\right)
\neq0.
\]

\begin{lemma}\label{lem:graded-piece-map}
There exists a multi-index
$\mathbf c=(c_1,\ldots,c_n)\in\mathcal C_m
$ such that
\[
\operatorname{Hom}_{G_n(\R)}
\left(
\sigma_{\mathbf c},
\lvert\det\rvert_\mathbb R^{-s_0}
\right)
\neq0.
\]
\end{lemma}

\begin{proof}
By assumption, there exists
\[
0\neq\ell\in
\operatorname{Hom}_{G_n(\R)}
\left(
M_N,
\lvert\det\rvert_\mathbb R^{-s_0}
\right).
\]
Choose the smallest index $i\geq1$ such that
$\ell|_{M_i}\neq0$. Then
$M_{i-1}\subseteq\ker(\ell),$
so $\ell|_{M_i}$ descends to a nonzero continuous
$G_n(\R)$-equivariant functional on
\[
M_i/M_{i-1}\cong\sigma_{\mathbf c^{(i)}}.\qedhere
\]
\end{proof}

For an involution $w\in S_n$, put
\[
I_w=
\left\{
(i,j):i>j,\ w(i)>w(j)
\right\}.
\]
For a function $\kappa:I_w\longrightarrow\mathbb Z_{\geq0}$, define
\[
q_i(\kappa)
=
\sum_{\substack{j:(i,j)\in I_w}}\kappa(i,j)
-
\sum_{\substack{j:(j,i)\in I_w}}\kappa(j,i),
\qquad 1\leq i\leq n.
\]

\begin{lemma}\label{lem:asai-orbit-data}
Let $\mathbf c=(c_1,\ldots,c_n)\in\mathcal C_m$, and suppose that
\[
\operatorname{Hom}_{G_n(\R)}
\left(
\sigma_{\mathbf c},
\lvert\det\rvert_\mathbb R^{-s_0}
\right)
\neq0.
\]
Then there exist an involution $w\in S_n$ and a function
$\kappa:I_w\longrightarrow\mathbb Z_{\geq0}$ such that, upon writing
$q_i=q_i(\kappa)$, we have
\begin{equation}\label{eq:orbit-relation}
s_0+a_i+b_{w(i)}=q_i-c_i,
\qquad 1\leq i\leq n.
\end{equation}
Moreover,
\begin{equation}\label{eq:q-properties}
\sum_{i=1}^{n}q_i=0,
\qquad
q_1\leq0,
\end{equation}
and, for every fixed point $i$ of $w$,
\begin{equation}\label{eq:fixed-parity-q}
q_i\equiv m_i+c_i\pmod2.
\end{equation}
\end{lemma}

\begin{proof}
Put
\[
\eta_i
=
\chi_{m_i+c_i,\,s_i+\frac{c_i}{2}+\frac{s_0}{2}},
\qquad
\eta=\eta_1\boxtimes\cdots\boxtimes\eta_n.
\]
Then
\[
\eta_1\times\cdots\times\eta_n
\cong
\sigma_{\mathbf c}\otimes
\left(\chi_{0,s_0/2}\circ\det\right).
\]
Since the restriction of $\chi_{0,s_0/2}$ to $\mathbb R^\times$ is
$\lvert\>\cdot\>\rvert_{\mathbb R}^{s_0}$, the hypothesis gives
\[
\operatorname{Hom}_{G_n(\mathbb R)}
\left(\eta_1\times\cdots\times\eta_n,\mathbb C\right)
\neq0.
\]

We apply the orbit stratification and Frobenius reciprocity argument
of \cite[Section~4]{Kem2015}. The $B_n(\mathbb C)$-orbits on
$G_n(\mathbb C)/G_n(\mathbb R)$ are indexed by involutions in $S_n$.
For an involution $w$, the diagonal subgroup of the corresponding
stabilizer is
\[
A_n^w
=
\left\{
\operatorname{diag}(t_1,\ldots,t_n)\in A_n(\mathbb C):
t_i=\overline{t_{w(i)}}\text{ for every }i
\right\}.
\]
The nonvanishing above implies that, for some involution $w$ and some
integer $k\geq0$, the character $\eta|_{A_n^w}$ occurs as a weight in
the $k$-th symmetric power of the complexified normal space to the
corresponding orbit. By the normal-space calculation in
\cite[Lemma~4.1 and Corollary~4.2]{Kem2015}, this weight is the
restriction of a character
\[
\alpha_\kappa\bigl(\operatorname{diag}(t_1,\ldots,t_n)\bigr)
=
\prod_{(i,j)\in I_w}
\left(\frac{t_i}{t_j}\right)^{\kappa(i,j)}
=
\prod_{i=1}^n t_i^{q_i},
\]
where $\kappa:I_w\to\mathbb Z_{\geq0}$ satisfies
$\sum_{(i,j)\in I_w}\kappa(i,j)=k$ and $q_i=q_i(\kappa)$.
Thus
\[
\prod_{i=1}^n t_i^{q_i}
=
\prod_{i=1}^n\eta_i(t_i),
\qquad
\operatorname{diag}(t_1,\ldots,t_n)\in A_n^w.
\]

Suppose first that $w(i)=i$. For every $x\in\mathbb R^\times$, the
diagonal element with $x$ in the $i$-th position and $1$ elsewhere
belongs to $A_n^w$. Evaluating the preceding identity at this element
gives
\[
x^{q_i}
=
\eta_i(x)
=
\operatorname{sgn}(x)^{m_i+c_i}
|x|_{\mathbb R}^{\,2s_i+c_i+s_0}.
\]
Taking $x>0$ and varying $x$, we obtain
\[
q_i=2s_i+c_i+s_0=s_0+a_i+b_i+c_i,
\]
which proves \eqref{eq:orbit-relation} at the fixed point $i$.
Taking $x=-1$ gives
\[
(-1)^{q_i}=(-1)^{m_i+c_i},
\]
and hence also proves \eqref{eq:fixed-parity-q}.

We next consider a two-element orbit $\{i,j\}$ of $w$, so that
$j=w(i)\neq i$. For every $z\in\mathbb C^\times$, the diagonal
element with $z$ in position $i$, $\overline z$ in position $j$,
and $1$ elsewhere belongs to $A_n^w$. Therefore
\[
z^{q_i}\overline z^{\,q_j}
=
\eta_i(z)\eta_j(\overline z).
\]
Using
\[
z^r=\chi_{r,r/2}(z),
\qquad
\chi_{r,t}(\overline z)=\chi_{-r,t}(z),
\]
we can write this as the following equality of characters of
$\mathbb C^\times$:
\[
\chi_{q_i-q_j,\,\frac{q_i+q_j}{2}}
=
\chi_{m_i+c_i-m_j-c_j,\,
       s_i+s_j+\frac{c_i+c_j}{2}+s_0}.
\]
By uniqueness of the parameters in the expression $\chi_{r,t}$,
\[
\begin{aligned}
q_i-q_j&=m_i+c_i-m_j-c_j,\\
q_i+q_j&=2s_i+2s_j+c_i+c_j+2s_0.
\end{aligned}
\]
Adding these two equations and dividing by $2$, we find
\[
q_i
=s_0+s_i+s_j+\frac{m_i-m_j}{2}+c_i
=s_0+a_i+b_j+c_i.
\]
Interchanging $i$ and $j$ gives
$q_j=s_0+a_j+b_i+c_j$.
Since $j=w(i)$ and $i=w(j)$, this proves
\eqref{eq:orbit-relation} for both indices.

Finally, each $\kappa(i,j)$ contributes once with a positive sign
and once with a negative sign to $\sum_i q_i$. Hence
\[
\sum_{i=1}^n q_i
=
\sum_{(i,j)\in I_w}\kappa(i,j)
-
\sum_{(i,j)\in I_w}\kappa(i,j)
=0.
\]
Moreover, no pair of the form $(1,j)$ belongs to $I_w$. Thus
\[
q_1
=-\sum_{j:(j,1)\in I_w}\kappa(j,1)
\leq0.
\]
This proves \eqref{eq:q-properties} and completes the proof.
\end{proof}

\begin{prop}\label{principal-cong-inverse}
Let
\[
    \pi = \chi_{m_1,s_1} \times \cdots \times \chi_{m_n,s_n}
\]
be in general position. Let $s_0 \in \mathbb{C}$ and $m \in \mathbb{Z}_{\geq 0}$. Then
\[
    \operatorname{Hom}_{G_n(\R)} \left(
        \pi \widehat{\otimes} \Sym^m(\mathbb{C}^n), \lvert\det\rvert_\mathbb R^{-s_0}
    \right) \neq 0
\]
if and only if there exist integers
\[
    c_1, \ldots, c_n \geq 0, \qquad \sum_{i=1}^{n} c_i = m,
\]
and an involution $w \in S_n$ such that
\begin{equation}\label{eq:asai-congruence-inverse}
    \begin{aligned}
        m_{w(i)} + c_{w(i)} &= m_i + c_i, && \quad 1 \leq i \leq n, \\
        s_{w(i)} + s_i + \frac{c_{w(i)} + c_i}{2} + s_0 &= 0, && \quad 1 \leq i \leq n,
    \end{aligned}
\end{equation}
and
\begin{equation}\label{eq:asai-fixed-parity}
    m_i + c_i \equiv 0 \pmod 2 \qquad \text{whenever } w(i) = i.
\end{equation}
\end{prop}
\begin{proof}
Suppose first that the Hom space in the statement is nonzero. By Lemma~\ref{lem:graded-piece-map}, there exists a multi-index
$\mathbf c=(c_1,\ldots,c_n)\in\mathcal C_m$ such that
\[
\operatorname{Hom}_{G_n(\R)}
\left(
\sigma_{\mathbf c},
\lvert\det\rvert_\mathbb R^{-s_0}
\right)
\neq0.
\]
The holomorphic and antiholomorphic parameters of $\sigma_{\mathbf c}$ are
\[
a_i+c_i
\qquad\text{and}\qquad
b_i,
\qquad 1\leq i\leq n.
\]
Since $\pi$ is in general position, so is $\sigma_{\mathbf c}$; in particular,
$\sigma_{\mathbf c}$ is irreducible and generic.

Put
\[
\Sigma_{\mathbf c}
=
\sigma_{\mathbf c}\otimes
\left(\chi_{0,s_0/2}\circ\det\right).
\]
Twisting preserves irreducibility and genericity, so
$\Sigma_{\mathbf c}$ is irreducible and generic. Further, $\Sigma_{\mathbf c}$ is $G_n(\R)$-distinguished. For a character $\chi$ of $\mathbb C^\times$, we write
$\overline{\chi}(z)=\chi(\overline z)$. By
\cite[Theorem~1.1]{PWZ2026}, there exists an involution $w\in S_n$ such that
\[
\chi_{m_{w(i)}+c_{w(i)},\,
      s_{w(i)}+\frac{c_{w(i)}}2+\frac{s_0}{2}}
=
\overline{
\chi_{m_i+c_i,\,
      s_i+\frac{c_i}{2}+\frac{s_0}{2}}
}^{\,-1}
\]
for every $1\leq i\leq n$, and
\[
\chi_{m_i+c_i,\,
      s_i+\frac{c_i}{2}+\frac{s_0}{2}}(-1)=1
\]
whenever $w(i)=i$.  Since
\[
\overline{\chi_{r,t}}^{\,-1}=\chi_{r,-t},
\]
these conditions are precisely
\eqref{eq:asai-congruence-inverse} and
\eqref{eq:asai-fixed-parity}.

Conversely, suppose that there exist integers
$c_1,\ldots,c_n\geq0$, with $\sum_{i=1}^{n}c_i=m$, and an involution
$w\in S_n$ satisfying
\eqref{eq:asai-congruence-inverse} and
\eqref{eq:asai-fixed-parity}. Put
$\mathbf c=(c_1,\ldots,c_n)$. As above, $\sigma_{\mathbf c}$ is irreducible and generic. The conditions
\eqref{eq:asai-congruence-inverse} and
\eqref{eq:asai-fixed-parity}, together with
\cite[Theorem~1.1]{PWZ2026}, give
\[
\operatorname{Hom}_{G_n(\R)}
\left(
\sigma_{\mathbf c},
\lvert\det\rvert_\mathbb R^{-s_0}
\right)
\neq0.
\]

Put $\mathcal V=\pi\widehat\otimes\Sym^m(\mathbb C^n)$.
For $\mathbf d=(d_1,\ldots,d_n)\in\mathcal C_m$, write
$\sigma_{\mathbf d}$ for the corresponding graded quotient.
Let $j$ be the index for which
\[
M_j/M_{j-1}\cong\sigma_{\mathbf c},
\]
and let $\zeta$ be the infinitesimal character of $\sigma_{\mathbf c}$.

We claim that
\[
\sigma_{\mathbf d}[\zeta]=0
\qquad
\text{whenever }\mathbf d\neq\mathbf c.
\]
Indeed, the infinitesimal character of $\sigma_{\mathbf d}$ is represented by
the pair of multisets
\[
\left(
\{a_i+d_i:1\leq i\leq n\},
\{b_i:1\leq i\leq n\}
\right),
\]
whereas that of $\sigma_{\mathbf c}$ is represented by
\[
\left(
\{a_i+c_i:1\leq i\leq n\},
\{b_i:1\leq i\leq n\}
\right).
\]
If these infinitesimal characters were equal, there would be a permutation
$u\in S_n$ such that
\[
a_i+d_i=a_{u(i)}+c_{u(i)},
\qquad 1\leq i\leq n.
\]
Thus $a_i-a_{u(i)}\in\mathbb Z$. The general-position assumption implies
$u(i)=i$ for every $i$, and consequently $d_i=c_i$ for every $i$.
This proves the claim.

By Lemma~\ref{lem:primary-decomposition} and exactness along the filtration,
the claim gives
\[
M_{j-1}[\zeta]=0
\qquad\text{and}\qquad
(\mathcal V/M_j)[\zeta]=0.
\]
Applying $X\mapsto X[\zeta]$ to
\[
0\longrightarrow M_{j-1}\longrightarrow M_j
\longrightarrow\sigma_{\mathbf c}\longrightarrow0
\]
and
\[
0\longrightarrow M_j\longrightarrow\mathcal V
\longrightarrow\mathcal V/M_j\longrightarrow0,
\]
we obtain
\[
\mathcal V[\zeta]\cong\sigma_{\mathbf c}.
\]
Composing the primary projection with this isomorphism gives a
$G_n(\C)$-equivariant surjection
\[
\mathcal V
\longtwoheadrightarrow
\mathcal V[\zeta]
\cong
\sigma_{\mathbf c}.
\]
Composing this map with a nonzero element of
\[
\operatorname{Hom}_{G_n(\R)}
\left(
\sigma_{\mathbf c},
\lvert\det\rvert_\mathbb R^{-s_0}
\right)
\]
gives
\[
\operatorname{Hom}_{G_n(\R)}
\left(
\pi\widehat\otimes\Sym^m(\mathbb C^n),
\lvert\det\rvert_\mathbb R^{-s_0}
\right)
\neq0.
\]
This proves the converse.
\end{proof}

We now show that the nonvanishing of the Hom space in
Theorem~\ref{1.1}~$(ii)$ implies that $s_0$ is a pole of the
Asai $L$-function and give a lower bound for its order.
In general position, we compute the order exactly.

\begin{prop}\label{pole}
Let
\[
\pi=
\chi_{m_1,s_1}\times\cdots\times\chi_{m_n,s_n}
\]
be an irreducible generic representation of $G_n(\C)$. Let
$s_0\in\mathbb C$ and $m\in\mathbb Z_{\geq0}$, and suppose that
\[
\operatorname{Hom}_{G_n(\R)}
\left(
\pi\widehat{\otimes}\Sym^m(\mathbb C^n),
\lvert\det\rvert_\mathbb R^{-s_0}
\right)
\neq0.
\]
After reordering the inducing characters, assume that
\[
m_1\geq m_2\geq\cdots\geq m_n.
\]
Then there exist a multi-index
$\mathbf c=(c_1,\ldots,c_n)\in\mathcal C_m$, an involution $w\in S_n$,
and a function $\kappa:I_w\to\mathbb Z_{\geq0}$ as in
Lemma~\ref{lem:asai-orbit-data}. Put
\[
q_i=q_i(\kappa),
\qquad
d_i=c_i-q_i,
\qquad 1\leq i\leq n,
\]
and set
\begin{equation}\label{eq:asai-pole-lower-bound}
D(\mathbf c,w,\kappa)
=
\#\left\{i:w(i)=i,\ d_i\geq0\right\}
+
\#\left\{i:i<w(i),\ d_i\geq0\right\}.
\end{equation}
Then $L(s,\pi,\operatorname{As})$ has a pole at $s_0$ of order at least $D(\mathbf c,w,\kappa)\geq1.$

If, in addition, $\pi$ is in general position and
$\mathbf c=(c_1,\ldots,c_n)$ and $w$ are as in
Proposition~\ref{principal-cong-inverse}, then the order of the pole is
\begin{equation}\label{eq:asai-pole-order-gp}
\begin{aligned}
D_{\mathrm{gp}}
&=
\#\{i:w(i)=i\}+\#\{i:i<w(i)\}\\
&=
\frac{n+\#\operatorname{Fix}(w)}{2}
=
\left\lceil\frac n2\right\rceil.
\end{aligned}
\end{equation}
\end{prop}

\begin{proof}
Because $\pi$ is irreducible, we can reorder its inducing characters without changing the isomorphism class of
$\pi$; see \cite{SV1980}. We therefore use the realization for which
\[
m_1\geq m_2\geq\cdots\geq m_n.
\]
By Lemma~\ref{lem:graded-piece-map}, there exists
$\mathbf c=(c_1,\ldots,c_n)\in\mathcal C_m$ such that
\[
\operatorname{Hom}_{G_n(\R)}
\left(
\sigma_{\mathbf c},
\lvert\det\rvert_\mathbb R^{-s_0}
\right)
\neq0.
\]
Lemma~\ref{lem:asai-orbit-data} gives an involution $w\in S_n$ and a
function $\kappa:I_w\to\mathbb Z_{\geq0}$ such that, with
$q_i=q_i(\kappa)$,
\[
s_0+a_i+b_{w(i)}=q_i-c_i=-d_i.
\]

For each $1\leq i\leq n$, let
$\varepsilon_i\in\{0,1\}$ be determined by
\[
\varepsilon_i\equiv m_i\pmod2.
\]
By the local Langlands correspondence \cite{Kna1994} and the direct-sum
formula for Asai representations \cite[Section~2.2]{JY2026}, we have
\begin{equation}\label{eq:asai-L-principal-series}
\begin{aligned}
L(s,\pi,\operatorname{As})
={}&
\prod_{i=1}^{n}
\Gamma_{\mathbb R}
\left(s+2s_i+\varepsilon_i\right)\\
&\times
\prod_{1\leq i<j\leq n}
\Gamma_{\mathbb C}
\left(
 s+s_i+s_j+\frac{|m_i-m_j|}{2}
\right).
\end{aligned}
\end{equation}
Recall that $\Gamma_{\mathbb R}(z)$ has simple poles at
$2\mathbb Z_{\leq0}$, whereas $\Gamma_{\mathbb C}(z)$ has simple poles at
$\mathbb Z_{\leq0}$. Since Gamma functions have no zeros, these poles cannot
cancel. Thus, we inspect the factors at $s_0$.

\begin{enumerate}
\item[(i)]
\textit{Fixed points of $w$.}
Suppose that $w(i)=i$ and $d_i\geq0$. Then
\[
s_0+2s_i=-d_i.
\]
Moreover, \eqref{eq:fixed-parity-q} gives
\[
d_i=c_i-q_i\equiv m_i\equiv\varepsilon_i\pmod2.
\]
Consequently,
\[
s_0+2s_i+\varepsilon_i
=-d_i+\varepsilon_i
\in2\mathbb Z_{\leq0}.
\]
Thus the $i$-th diagonal $\Gamma_{\mathbb R}$-factor contributes one
simple pole.

\item[(ii)]
\textit{Two-element orbits of $w$.}
Suppose that $i<w(i)$ and $d_i\geq0$. Since the inducing characters have
been ordered so that $m_i\geq m_{w(i)}$, the corresponding
$\Gamma_{\mathbb C}$-factor in \eqref{eq:asai-L-principal-series} is
\[
\Gamma_{\mathbb C}
\left(s+a_i+b_{w(i)}\right).
\]
At $s_0$, its argument is
\[
s_0+a_i+b_{w(i)}=-d_i\in\mathbb Z_{\leq0}.
\]
Hence this factor contributes one simple pole.

\item[(iii)]
\textit{Positivity of the lower bound.}
By \eqref{eq:q-properties}, we have $q_1\leq0$, and therefore
\[
d_1=c_1-q_1\geq0.
\]
If $w(1)=1$, the index $1$ is counted in the first term of
\eqref{eq:asai-pole-lower-bound}. Otherwise, $1<w(1)$, so it is
counted in the second term. Hence $D(\mathbf c,w,\kappa)\geq1$.
\end{enumerate}

The factors identified in (i) and (ii) are distinct. Together with
the positivity established in (iii), this shows that
$L(s,\pi,\operatorname{As})$ has a pole at $s_0$ of order at least
$D(\mathbf c,w,\kappa)$.

Assume now that $\pi$ is in general position, and let $\mathbf c$ and $w$ be
as in Proposition~\ref{principal-cong-inverse}. We first determine the factors
that have a pole at $s_0$. If $w(i)=i$, then
\[
s_0+2s_i=-c_i
\qquad\text{and}\qquad
c_i\equiv m_i\equiv\varepsilon_i\pmod2.
\]
Therefore
\[
s_0+2s_i+\varepsilon_i
=-c_i+\varepsilon_i
\in2\mathbb Z_{\leq0},
\]
and the $i$-th diagonal factor has a simple pole.

If $i<w(i)$, then \eqref{eq:asai-congruence-inverse} gives
\[
m_i+c_i=m_{w(i)}+c_{w(i)}
\]
and
\[
s_0+s_i+s_{w(i)}
=-\frac{c_i+c_{w(i)}}{2}.
\]
Consequently,
\[
\begin{aligned}
s_0+s_i+s_{w(i)}
+\frac{|m_i-m_{w(i)}|}{2}
&=
-\frac{c_i+c_{w(i)}}{2}
+\frac{|c_i-c_{w(i)}|}{2}\\
&=-\min\{c_i,c_{w(i)}\}
\in\mathbb Z_{\leq0}.
\end{aligned}
\]
Thus the factor attached to the pair $\{i,w(i)\}$ has a simple pole.
This gives the lower bound $D_{\mathrm{gp}}$ in
\eqref{eq:asai-pole-order-gp}.

It remains to show that no other factor has a pole. Suppose first that the
$i$-th diagonal factor has a pole and that $w(i)\neq i$. Then
\[
s_0+a_i+b_i\in\mathbb Z.
\]
On the other hand, \eqref{eq:asai-congruence-inverse} gives
\[
s_0+a_i+b_{w(i)}=-c_i\in\mathbb Z.
\]
It follows that
\[
b_i-b_{w(i)}\in\mathbb Z,
\]
contrary to the general-position assumption.

Finally, let $i<j$ with $j\neq w(i)$. Since $m_i\geq m_j$,
the corresponding factor is
\[
\Gamma_{\mathbb C}(s+a_i+b_j).
\]
If this factor had a pole at $s_0$, then
$s_0+a_i+b_j\in\mathbb Z$. Since
\[
s_0+a_i+b_{w(i)}=-c_i\in\mathbb Z,
\]
we would obtain
\[
b_j-b_{w(i)}\in\mathbb Z.
\]
The general-position assumption would then imply $j=w(i)$, contrary to the
choice of the pair. Hence no other off-diagonal factor contributes a pole.

Adding the contributions, the order of the pole at $s_0$ is
$D_{\mathrm{gp}}$. To obtain the last equality in
\eqref{eq:asai-pole-order-gp}, observe that a fixed point $i$
of $w$ satisfies
\[
2s_i+c_i+s_0=0,
\qquad
m_i+c_i\equiv0\pmod2.
\]
Consequently,
\[
a_i=s_i+\frac{m_i}{2}
=
-\frac{s_0}{2}+\frac{m_i-c_i}{2}
\in-\frac{s_0}{2}+\mathbb Z.
\]
Two distinct fixed points would therefore contradict the
condition $a_i-a_j\notin\mathbb Z$. Thus $w$ has at most one
fixed point. Since
$\#\operatorname{Fix}(w)\equiv n\pmod2$, this number is $0$
when $n$ is even and $1$ when $n$ is odd. Hence
\[
D_{\mathrm{gp}}
=
\frac{n+\#\operatorname{Fix}(w)}2
=
\left\lceil\frac n2\right\rceil.
\]
\end{proof}
\end{section}

\begin{section}{Proof of the Main Theorem}\label{s4}

For $\pi\in\Irr(G_n(\C))$ and $s\in\C$, put
\begin{equation}\label{eq:Hspace}
\mathcal H_s(\pi)
=
\operatorname{Hom}_{G_n(\R)}
\left(
\pi\widehat{\otimes}\mathcal S,
\lvert\det\rvert_\mathbb R^{-s}
\right).
\end{equation}

\begin{prop}[All-parameter uniqueness]\label{prop:all-s-uniqueness}
Let $\pi\in\Irr(G_n(\C))$. For every $s\in\C$, we have
\[
 \dim\mathcal H_s(\pi)\le1.
\]
This multiplicity bound also holds if $\lvert\det\rvert_\mathbb R^{-s}$ is replaced by any smooth character of $G_n(\R)$.
\end{prop}

\begin{proof}
Let $\chi$ be a smooth character of $G_n(\R)$. By
\cite[Theorem~1.4 and Remark~3.3]{CSW2026}, applied to
$\chi^{-1}$, we have
\begin{equation}\label{eq:CSW-theorem}
\dim\operatorname{Hom}_{G_n(\R)}
\left(
\pi\widehat{\otimes}(\mathcal S\otimes\chi^{-1}),
\mathbb C
\right)
\leq1.
\end{equation}
The action on $\mathcal S$ in that result is
$(g\cdot\Phi)(x)=\Phi(xg)$, exactly as in
\eqref{eq:schwartz-row-action-asai}. The Hom space in
\eqref{eq:CSW-theorem} is canonically isomorphic to
\[
\operatorname{Hom}_{G_n(\R)}
\left(\pi\widehat{\otimes}\mathcal S,\chi\right).
\]
Taking $\chi=\lvert\det\rvert_\mathbb R^{-s}$ proves the assertion.
\end{proof}

For the remainder of this section, let
$\pi\in\Irr_{\mathrm{gen}}(G_n(\C))$, and fix a
$G_n(\C)$-equivariant topological identification
\[
\pi\simeq\mathcal W(\pi,\psi).
\]
Recall the normalized Flicker functional
\begin{equation}\label{eq:Lambda-def}
\Lambda_s(W,\Phi)
=
\frac{I(s,W,\Phi)}{L(s,\pi,\operatorname{As})}.
\end{equation}
For every fixed $W$ and $\Phi$, this quotient extends to an entire
function of $s$ by \cite[Theorem~3.5.2]{BP2021}. By \cite[Theorem~6.1]{JY2026}, for each fixed $s \in \mathbb{C}$, $\Lambda_s$ is a continuous bilinear functional satisfying the equivariance property
\begin{equation}\label{eq:Lambda-equivariance}
\Lambda_s(g\cdot W,g\cdot\Phi)
=
\lvert\det g\rvert_{\mathbb R}^{-s}\Lambda_s(W,\Phi),
\qquad g\in G_n(\mathbb R).
\end{equation}
Consequently, $\Lambda_s\in\mathcal{H}_s(\pi)$. Furthermore, the same theorem guarantees that $\Lambda_s \neq 0$ for all $s \in \mathbb{C}$.

We realize $\operatorname{Sym}^m(\mathbb C^n)$ as the space of
homogeneous polynomials of degree $m$ in a row variable, with
the action
\[
(g\cdot P)(x)=P(xg).
\]
For $m\in\mathbb Z_{\geq0}$, define
\begin{equation}\label{eq:jet-map}
J_m:\mathcal S\longrightarrow\operatorname{Sym}^m(\mathbb C^n)
\end{equation}
by taking the homogeneous degree-$m$ Taylor term at the origin:
\[
J_m\Phi
=
\sum_{|\alpha|=m}
\frac{\partial^\alpha\Phi(0)}{\alpha!}X^\alpha.
\]
The chain rule shows that $J_m$ is $G_n(\R)$-equivariant. It is
continuous and surjective, and its restriction to the
filtration satisfies
\begin{equation}\label{eq:jet-properties}
J_m(\mathcal{S}^{m+1})=0,
\qquad
J_m|_{\mathcal{S}^m}:
\mathcal{S}^m\longtwoheadrightarrow
\operatorname{Sym}^m(\mathbb C^n),
\qquad
\ker(J_m|_{\mathcal{S}^m})
=
\mathcal{S}^{m+1}.
\end{equation}
Indeed, surjectivity follows by multiplying a homogeneous
polynomial of degree $m$ by a compactly supported smooth
function which is identically $1$ near the origin.

\begin{lemma}[The jet form]\label{lem:jet-form}
Let $s_0\in\mathbb C$ and $m\in\mathbb Z_{\geq0}$, and suppose that
\[
\operatorname{Hom}_{G_n(\R)}
\left(
\pi\widehat{\otimes}\operatorname{Sym}^m(\mathbb C^n),
\lvert\det\rvert_\mathbb R^{-s_0}
\right)
\neq0.
\]
Then there exists $T_m\in\mathcal H_{s_0}(\pi)$ that
vanishes identically on
$\mathcal W(\pi,\psi)\times\mathcal S^{m+1}$
but is not identically zero on
$\mathcal W(\pi,\psi)\times\mathcal S^m$.
\end{lemma}

\begin{proof}
Choose
\[
0\neq\ell_m\in
\operatorname{Hom}_{G_n(\R)}
\left(
\pi\widehat{\otimes}\operatorname{Sym}^m(\mathbb C^n),
\lvert\det\rvert_\mathbb R^{-s_0}
\right),
\]
and, using the fixed identification of $\pi$ with its Whittaker
model, define
\begin{equation}\label{eq:Tm-def}
T_m(W,\Phi)=\ell_m\bigl(W\otimes J_m\Phi\bigr).
\end{equation}
The continuity and equivariance of $J_m$ show that
$T_m\in\mathcal H_{s_0}(\pi)$. By
\eqref{eq:jet-properties}, $J_m$ vanishes on $\mathcal S^{m+1}$,
so $T_m$ vanishes on
$\mathcal W(\pi,\psi)\times\mathcal S^{m+1}$. Since
$J_m|_{\mathcal S^m}$ is surjective and $\ell_m\neq0$, the
restriction of $T_m$ to
$\mathcal W(\pi,\psi)\times\mathcal S^m$ is nonzero.
\end{proof}

\begin{proof}[Proof of Theorem~\ref{1.1}]
The implication $(i)\Rightarrow(ii)$ was proved at the end of
Section~\ref{s2}. Assume $(ii)$. By Proposition~\ref{pole},
$L(s,\pi,\operatorname{As})$ has a pole at $s_0$; let $d\geq1$
be its order. By the discussion in Section~\ref{s1}, $d$ is
also the maximal pole order of $\mathcal J(\pi)$ at $s_0$.

By Lemma~\ref{lem:jet-form}, there exists a nonzero
$T_m\in\mathcal H_{s_0}(\pi)$ that vanishes on
$\mathcal W(\pi,\psi)\times\mathcal S^{m+1}$ but not on
$\mathcal W(\pi,\psi)\times\mathcal S^m$.
Since $\Lambda_{s_0}$ is also a nonzero element of
$\mathcal H_{s_0}(\pi)$, Proposition~\ref{prop:all-s-uniqueness}
gives
\begin{equation}\label{eq:Lambda-is-jet}
\Lambda_{s_0}=cT_m
\qquad\text{for some }c\in\mathbb C^\times.
\end{equation}

Put
\[
a_{s_0}
=
\lim_{s\to s_0}(s-s_0)^dL(s,\pi,\operatorname{As})
\neq0.
\]
By \eqref{eq:Lambda-def} and \eqref{eq:Lambda-is-jet},
\[
\begin{aligned}
B_{s_0}(W,\Phi)
&=\lim_{s\to s_0}(s-s_0)^dI(s,W,\Phi)\\
&=a_{s_0}\Lambda_{s_0}(W,\Phi)
=c\,a_{s_0}T_m(W,\Phi).
\end{aligned}
\]
Since $c\,a_{s_0}\neq0$, the form $B_{s_0}$ has the same
vanishing properties as $T_m$. Thus $s_0$ is an exceptional
pole of level $m$ for $\mathcal J(\pi)$.
This proves $(ii)\Rightarrow(i)$ and completes the proof.
\end{proof}

\end{section}

\begin{section}{Archimedean Derivatives and Exceptional L-factors}\label{s5}

In this section, we prove Theorem~\ref{1.2}. First, we record
the following characterization of exceptional poles.

\begin{thm}\label{classification-general}
Let
\[
\pi=
\chi_{m_1,s_1}\times\cdots\times\chi_{m_n,s_n}
\]
be in general position. Let
\(s_0\in\mathbb C\) and \(m\in\mathbb Z_{\geq0}\).
The following conditions are equivalent:
\begin{enumerate}[label=$(\roman*)$]
    \item
    $s_0$ is an exceptional pole of level $m$ for $\pi$.

    \item
    There exist integers
    \[
    c_1,\ldots,c_n\geq0,
    \qquad
    \sum_{i=1}^{n}c_i=m,
    \]
    and an involution $w\in S_n$ such that
    \begin{equation}\label{eq:asai-classification}
    \begin{aligned}
    m_{w(i)}+c_{w(i)}
    &=m_i+c_i,
    &&1\leq i\leq n,\\
    s_{w(i)}+s_i+\frac{c_{w(i)}+c_i}{2}+s_0
    &=0,
    &&1\leq i\leq n,
    \end{aligned}
    \end{equation}
    and
    \begin{equation}\label{eq:asai-classification-fixed}
    m_i+c_i\equiv0\pmod2
    \qquad\text{whenever }w(i)=i.
    \end{equation}

    \item
    $\operatorname{Hom}_{\operatorname{GL}_n(\R)}
    \left(
    \pi\widehat{\otimes}\operatorname{Sym}^m(\C^n),
    \lvert\det\rvert_\mathbb R^{-s_0}
    \right)\neq0$.
\end{enumerate}
\end{thm}

\begin{proof}
The equivalence of\/ $(ii)$ and $(iii)$ is
Proposition~\ref{principal-cong-inverse}. The equivalence of\/ $(i)$ and $(iii)$ is
Theorem~\ref{1.1}.
\end{proof}

\begin{corollary}\label{cor:exceptional-pole-order}
Let $\pi$ be as in Theorem~\ref{classification-general}.
Every exceptional pole of $\pi$ has order
$\left\lceil n/2\right\rceil$.
\end{corollary}

\begin{proof}
This follows from Theorem~\ref{1.1} and
Proposition~\ref{pole}.
\end{proof}

For the remainder of this section, let $\pi$ be as in
Theorem~\ref{classification-general}, and put
\[
P_{\mathrm{ex}}
=
\{z\in\mathbb C:\text{$z$ is an exceptional pole of $\pi$
for some level $m\geq0$}\}.
\]
For a meromorphic function $F$, write $\operatorname{pord}_z(F)$
for its pole order at $z$, with value $0$ when $F$ is holomorphic
at $z$. Define
\[
d_{\mathrm{ex}}(z;\pi)=
\begin{cases}
\left\lceil n/2\right\rceil,&z\in P_{\mathrm{ex}},\\
0,&z\notin P_{\mathrm{ex}}.
\end{cases}
\]
By Corollary~\ref{cor:exceptional-pole-order}, this agrees with
the pole order of $L(s,\pi,\operatorname{As})$ at every
exceptional pole. For a representation of $G_r(\mathbb C)$
in general position, we use the same definitions with $n$
replaced by $r$.

\begin{prop}\label{exceptional-set}
There exist $r\geq0$ and $z_1,\ldots,z_r\in\mathbb C$, pairwise
incongruent modulo $2\mathbb Z$, such that
\[
P_{\mathrm{ex}}
=
\bigsqcup_{j=1}^{r}(z_j-2\mathbb Z_{\geq0}).
\]
\end{prop}

\begin{proof}
If $z$ satisfies Theorem~\ref{classification-general} with data
$(c_1,\ldots,c_n,w)$, then $z-2$ satisfies it with data
$(c_1+2,\ldots,c_n+2,w)$. Thus $z\in P_{\mathrm{ex}}$ implies
$z-2\in P_{\mathrm{ex}}$.

The product formula for $L(s,\pi,\operatorname{As})$, together
with $\Gamma_{\mathbb C}(s)=\Gamma_{\mathbb R}(s)
\Gamma_{\mathbb R}(s+1)$, shows that its poles lie in finitely
many progressions $b-2\mathbb Z_{\geq0}$. Thus $P_{\mathrm{ex}}$
is contained in finitely many such progressions. In each
congruence class that meets $P_{\mathrm{ex}}$, choose $z_j$
with maximal real part. The assertion follows from the
stability of $P_{\mathrm{ex}}$ under $z\mapsto z-2$.
\end{proof}

\begin{defn}\label{Lex}
With $z_1,\ldots,z_r$ as in Proposition~\ref{exceptional-set},
define the exceptional $L$-factor by
\[
L_{\mathrm{ex}}(s,\pi,\operatorname{As})
=
\prod_{j=1}^{r}
\Gamma_{\mathbb R}(s-z_j)^{\left\lceil n/2\right\rceil}.
\]
When $P_{\mathrm{ex}}=\varnothing$, the empty product is $1$.
\end{defn}

\begin{prop}\label{Lex-pole-orders}
For every $z\in\mathbb C$,
\[
\operatorname{pord}_z L_{\mathrm{ex}}(s,\pi,\operatorname{As})
=d_{\mathrm{ex}}(z;\pi).
\]
\end{prop}

\begin{proof}
Each $\Gamma_{\mathbb R}(s-z_j)$ has simple poles precisely at
$z_j-2\mathbb Z_{\geq0}$. These sets are pairwise disjoint,
and the Gamma function has no zeros.
\end{proof}

We next define the least common multiple of finite products of
inverse $\Gamma_{\mathbb R}$-factors. Throughout, such products
are understood to have scalar prefactor $1$. 

Consider a finite set $\mathcal{F}=\{F_1(s),\ldots,F_r(s)\}$, where
$$F_i(s) = \prod_j \Gamma_{\mathbb R}(s-a_{ij})^{-m_{ij}}$$
for some constants $a_{ij}\in\mathbb C$ and integers $m_{ij}\in\mathbb Z_{\geq0}$. Because $\Gamma_{\mathbb R}(s)^{-1}$ is an entire function, each $F_i$ is also entire. For $z\in\mathbb C$, let $\operatorname{ord}_z(F_i)$ denote the order of vanishing of $F_i$ at $s=z$. We then define
$$M_{\mathcal F}(z) = \max_{1\leq i\leq r} \operatorname{ord}_z(F_i) \quad \text{and} \quad c_{\mathcal F}(z) = M_{\mathcal F}(z)-M_{\mathcal F}(z+2).$$

\begin{lemma}\label{lem:lcm-jumps}
For every $z\in\mathbb C$, we have $c_{\mathcal F}(z)\in\mathbb Z_{\geq0}$. Moreover, $c_{\mathcal F}(z)\neq0$ for only finitely many $z\in\mathbb C$.
\end{lemma}

\begin{proof}
For every $i$ and $z\in\mathbb C$, the order of vanishing is
$$\operatorname{ord}_z(F_i) = \sum_{\substack{j\\ a_{ij}-z\in2\mathbb Z_{\geq0}}} m_{ij}.$$
Hence
$$\operatorname{ord}_z(F_i) - \operatorname{ord}_{z+2}(F_i) = \sum_{\substack{j\\a_{ij}=z}} m_{ij} \geq 0.$$
Since $\operatorname{ord}_z(F_i) \geq \operatorname{ord}_{z+2}(F_i)$ holds for all $i$, taking the maximum over $1 \leq i \leq r$ yields $M_{\mathcal F}(z)\geq M_{\mathcal F}(z+2)$, which proves that $c_{\mathcal F}(z)\in\mathbb Z_{\geq0}$.

Finally, if $z\notin\{a_{ij}\}_{i,j}$, then $\operatorname{ord}_z(F_i) = \operatorname{ord}_{z+2}(F_i)$ for all $i$. This implies $M_{\mathcal F}(z)=M_{\mathcal F}(z+2)$, and hence $c_{\mathcal F}(z)=0$. Thus, the support of $c_{\mathcal F}$ satisfies $\operatorname{Supp}(c_{\mathcal F}) \subseteq \{a_{ij}\}_{i,j}$, a finite set.
\end{proof}

\begin{defn}\label{lcm}
The \emph{least common multiple} of
$\mathcal F=\{F_1,\ldots,F_r\}$ is defined by
\[
\operatorname{l.c.m.}(F_1,\ldots,F_r)
=
\prod_{\substack{z\in\mathbb C\\c_{\mathcal F}(z)>0}}
\Gamma_{\mathbb R}(s-z)^{-c_{\mathcal F}(z)}.
\]
By Lemma~\ref{lem:lcm-jumps}, this is a finite product.
\end{defn}

\begin{prop}[Characterization of the l.c.m.;
  {\cite[Proposition~5.10]{JNY2026}}]
\label{prop:lcm-characterization}
Let
\[
\mathcal F=\{F_1,\ldots,F_r\}
\]
be as above, and put
\[
H(s)=\operatorname{l.c.m.}(F_1,\ldots,F_r).
\]
Then, for every $w\in\mathbb C$,
\[
\operatorname{ord}_w(H)
=
\max_{1\leq i\leq r}\operatorname{ord}_w(F_i).
\]
In particular, $H/F_i$ is entire for every $i$. Moreover, if $K$ is
another finite product of inverse $\Gamma_{\mathbb R}$-factors such
that $K/F_i$ is entire for every $i$, then $K/H$ is entire. Finally, $H$ is the unique finite product of inverse
$\Gamma_{\mathbb R}$-factors satisfying the displayed equality of
orders.
\end{prop}

We use the complex analogue of the archimedean derivative in
\cite[Section~3]{Chai2015}. For a smooth representation $(\pi,V)$,
a connected unipotent subgroup $U$, and a smooth character $\theta$
of $U$, write
\[
\pi_{U,\theta}
=
V\big/\overline{\operatorname{span}
\{\pi(u)v-\theta(u)v:u\in U,\ v\in V\}},
\qquad
\pi_U=\pi_{U,1}.
\]
Equivalently, since $U$ is connected, one may quotient by
the closed span of the Lie algebra relations
$d\pi(X)v-d\theta(X)v$, where
$X\in\operatorname{Lie}(U)$ and $v\in V$.

For $1\leq l\leq n$, put
\[
U_{n-l+1}
=
\left\{
\begin{pmatrix}I_{n-l}&x\\0&u\end{pmatrix}
:x\in\operatorname{Mat}_{n-l,l}(\mathbb C),\
u\in N_l(\mathbb C)
\right\},
\]
and define
\[
\theta_l\!\left(
\begin{pmatrix}I_{n-l}&x\\0&u\end{pmatrix}
\right)=\psi_l(u).
\]
Thus $U_1=N_n(\mathbb C)$. We regard $G_0(\mathbb C)$ as the
trivial group.

\begin{defn}[Archimedean derivative]
\label{def:archimedean-derivative}
For $\pi\in\Irr(G_n(\mathbb C))$, put $\pi^{(0)}=\pi$ and,
for $1\leq l\leq n$, define
\[
\pi^{(l)}
=
|\det|_{\mathbb C}^{-l/2}\otimes
\pi_{U_{n-l+1},\theta_l},
\]
where $G_{n-l}(\mathbb C)$ acts through
$g\mapsto\operatorname{diag}(g,I_l)$. This action is well defined
because $G_{n-l}(\mathbb C)$ normalizes $U_{n-l+1}$ and preserves
$\theta_l$.
\end{defn}

\begin{prop}[Explicit derivatives in general position]
\label{explicit-derivative}
Let $\pi=\chi_{m_1,s_1}\times\cdots\times\chi_{m_n,s_n}$
be in general position. For $I\subseteq\{1,\ldots,n\}$, put
\[
\pi_I=\mathop{\times}_{i\in I}\chi_{m_i,s_i},
\]
with the indices in increasing order, and put
$\pi_\varnothing=\mathbb C$.
For $0\leq q\leq n$, let
$$\mathscr I_q=\{I\subseteq\{1,\ldots,n\}:|I|=n-q\}.$$ Then
\[
\pi^{(q)}\simeq\bigoplus_{I\in\mathscr I_q}\pi_I.
\]
For $q<n$, the summands are in general
position and are pairwise nonisomorphic.
\end{prop}

\begin{proof}
The case $q=0$ follows from the definition. For $q=n$, genericity
and uniqueness of Whittaker functionals give
$\pi^{(n)}=\pi_{N_n(\mathbb C),\psi_n}\simeq\mathbb C$.

Assume $1\leq q<n$, and put
\[
Q=P_{n-q,q}(\mathbb C)=L_QU_Q,
\qquad
L_Q=G_{n-q}(\mathbb C)\times G_q(\mathbb C).
\]
Since $U_{n-q+1}=U_Q\rtimes N_q(\mathbb C)$ and $\theta_q$
is trivial on $U_Q$, taking coinvariants in stages gives
\begin{equation}\label{eq:derivative-homology-realization}
\pi^{(q)}
\simeq
|\det|_{\mathbb C}^{-q/2}\otimes
(\pi_{U_Q})_{N_q(\mathbb C),\psi_q}.
\end{equation}

Our general-position hypothesis implies \cite[(5.1.1)]{CC1999}.
By \cite[Theorem~5.3]{CC1999}, the zeroth nilradical homology of the
Harish--Chandra module of $\pi$ is a direct sum indexed by
$S_n/(S_{n-q}\times S_q)$. A class is determined by the subset
$I\in\mathscr I_q$ of inducing characters placed in the first
Levi factor; the remaining characters form the second factor.
Passing to smooth coinvariants by \cite[Theorem~1.5]{WY2026},
we obtain
\begin{equation}\label{eq:CC-decomposition-explicit}
\pi_{U_Q}
\simeq
\bigoplus_{I\in\mathscr I_q}
(\pi_I\widehat\otimes\pi_{I^c})\otimes\delta_Q^{1/2},
\end{equation}
where $I^c=\{1,\ldots,n\}\setminus I$. The modulus twist arises because induction is normalized,
whereas the action on coinvariants is unnormalized. Explicitly,
\[
\delta_Q^{1/2}(\operatorname{diag}(g_1,g_2))
=
|\det g_1|_{\mathbb C}^{q/2}
|\det g_2|_{\mathbb C}^{-(n-q)/2}.
\]

Both $\pi_I$ and $\pi_{I^c}$ are in general position. Therefore
$(\pi_{I^c})_{N_q(\mathbb C),\psi_q}\simeq\mathbb C.$
Since determinant characters are trivial on $N_q(\mathbb C)$,
\eqref{eq:CC-decomposition-explicit} gives
\[
(\pi_{U_Q})_{N_q(\mathbb C),\psi_q}
\simeq
\bigoplus_{I\in\mathscr I_q}
\left(\pi_I\otimes|\det|_{\mathbb C}^{q/2}\right).
\]
Substituting into \eqref{eq:derivative-homology-realization},
we conclude that
\[
\begin{aligned}
\pi^{(q)}
&\simeq
|\det|_{\mathbb C}^{-q/2}\otimes
\bigoplus_{I\in\mathscr I_q}
\left(\pi_I\otimes|\det|_{\mathbb C}^{q/2}\right)\\
&\simeq
\bigoplus_{I\in\mathscr I_q}\pi_I.
\end{aligned}
\]

Finally, for distinct $I,J\in\mathscr I_q$, the multisets
$\{a_i:i\in I\}$ and $\{a_j:j\in J\}$ differ, since the $a_i$
are pairwise distinct. Thus $\pi_I$ and $\pi_J$ have different
infinitesimal characters and are nonisomorphic.
\end{proof}
We now prove the pole-order comparison needed for Theorem~\ref{1.2}
directly from the factorization of the Asai $L$-function.

\begin{thm}\label{order-characterization}
Let
\[
\pi=
\chi_{m_1,s_1}\times\cdots\times\chi_{m_n,s_n}
\]
be in general position.

\begin{enumerate}[label=$(\roman*)$]
    \item\label{order-characterization-1}
    Suppose that $s_0$ is a pole of
    $L(s,\pi,\operatorname{As})$ of order $d$. Then $d\leq\left\lceil n/2\right\rceil$, and there
    exist an integer $q$ with
    $d\leq q\leq n$ and an irreducible constituent
    \[
    \sigma\subseteq\pi^{(n-q)}
    \]
    such that $s_0$ is an exceptional pole for
    $\sigma$ and
    \[
    \operatorname{pord}_{s_0}
    L(s,\sigma,\operatorname{As})=d.
    \]

    \item\label{order-characterization-2}
    Conversely, let $1\leq q\leq n$, and let
    \[
    \sigma\subseteq\pi^{(n-q)}
    \]
    be an irreducible constituent. If $s_0$ is an exceptional pole
    for $\sigma$ of order $d$, then
    \[
    \operatorname{pord}_{s_0}
    L(s,\pi,\operatorname{As})\geq d.
    \]
\end{enumerate}
\end{thm}

\begin{proof}
Let $\varepsilon_i\in\{0,1\}$ satisfy
$\varepsilon_i\equiv m_i\pmod2$. By
\eqref{eq:asai-L-principal-series}, the factors having a pole at
$s_0$ are indexed by
\[
D(s_0)=\{i:s_0+2s_i+\varepsilon_i\in2\mathbb Z_{\leq0}\}
\]
and
\[
E(s_0)=\left\{\{i,j\}:i<j,\quad
s_0+s_i+s_j+\frac{|m_i-m_j|}{2}\in\mathbb Z_{\leq0}\right\}.
\]
The members of
$\{\{i\}:i\in D(s_0)\}\cup E(s_0)$ are pairwise disjoint. Indeed, modulo $\mathbb Z$ every pole condition is
$s_0+a_i+a_j\equiv0$, with $i=j$ in the diagonal case,
since $a_j-b_j=m_j\in\mathbb Z$. Two distinct conditions
sharing an index would give $a_i-a_j\in\mathbb Z$ for
some $i\neq j$, contrary to general position. Moreover,
$i\in D(s_0)$ implies
$a_i\in-s_0/2+\mathbb Z$, so $\#D(s_0)\leq1$.

We prove~\ref{order-characterization-1}. Since Gamma factors
have simple poles and no zeros,
\begin{equation}\label{eq:asai-pole-order-count}
d=\#D(s_0)+\#E(s_0).
\end{equation}
Put
\begin{equation}\label{eq:asai-q-definition}
q=\#D(s_0)+2\#E(s_0),
\qquad
I=D(s_0)\cup\bigcup_{\{i,j\}\in E(s_0)}\{i,j\}.
\end{equation}
Then $|I|=q$ and $d\leq q\leq n$. Also
$2d=q+\#D(s_0)\leq n+1$, giving
$d\leq\lceil n/2\rceil$. By
Proposition~\ref{explicit-derivative},
\begin{equation}\label{eq:asai-selected-constituent}
\sigma=\pi_I\subseteq\pi^{(n-q)}.
\end{equation}

Define an involution $w$ on $I$ by fixing the elements of
$D(s_0)$ and interchanging each pair in $E(s_0)$.
For $i\in D(s_0)$, put
\begin{equation}\label{eq:asai-diagonal-c}
c_i=-s_0-2s_i\in\mathbb Z_{\geq0}.
\end{equation}
The diagonal pole condition also gives
$m_i+c_i\equiv0\pmod2$. For $\{i,j\}\in E(s_0)$, order
the indices so that $m_i\geq m_j$ and set
\[
r_{ij}=-s_0-s_i-s_j-\frac{m_i-m_j}{2}
\in\mathbb Z_{\geq0},
\]
\begin{equation}\label{eq:asai-offdiagonal-c}
c_i=r_{ij},\qquad c_j=r_{ij}+m_i-m_j.
\end{equation}
Then $m_i+c_i=m_j+c_j$ and
$s_0+s_i+s_j+(c_i+c_j)/2=0$. Thus
Theorem~\ref{classification-general}, applied in rank $q$,
shows that $s_0$ is an exceptional pole of level
$\sum_{i\in I}c_i$ for $\sigma$. Its pole order is $d$, since
the pole factors in $L(s,\sigma,\operatorname{As})$ are exactly
those indexed by $D(s_0)$ and $E(s_0)$.

We proceed to~\ref{order-characterization-2}. By
Proposition~\ref{explicit-derivative}, every constituent
$\sigma\subseteq\pi^{(n-q)}$ is induced from a subset of
$q$ of the characters defining $\pi$. The product formula gives
\[
L(s,\pi,\operatorname{As})
=L(s,\sigma,\operatorname{As})R(s),
\]
where $R$ is a product of Gamma factors. Since $R$ has no zeros,
the pole order for $\pi$ is at least that for $\sigma$.
\end{proof}

For $0\leq q<n$, write
\[
\pi^{(q)}
=
\bigoplus_i\pi_i^{(q)},
\]
where the direct sum is the indexed decomposition of
Proposition~\ref{explicit-derivative}.

\begin{corollary}\label{max-exceptional-order}
For every $s_0\in\mathbb C$,
\[
\operatorname{pord}_{s_0}
L(s,\pi,\operatorname{As})
=
\max_{\substack{0\leq q<n\\ i}}
d_{\mathrm{ex}}
\left(
s_0;
\pi_i^{(q)}
\right).
\]
\end{corollary}

\begin{proof}
The inequality
\[
\operatorname{pord}_{s_0}
L(s,\pi,\operatorname{As})
\geq
d_{\mathrm{ex}}
\left(
s_0;
\pi_i^{(q)}
\right)
\]
for every $q,i$ follows from
Theorem~\ref{order-characterization}\ref{order-characterization-2}. If $L(s,\pi,\operatorname{As})$ is holomorphic at $s_0$, then
Theorem~\ref{order-characterization}\ref{order-characterization-2}
shows that no derivative constituent can have an exceptional
pole at $s_0$. Hence every term in the maximum is zero. If it has a pole of order $d$, then
Theorem~\ref{order-characterization}\ref{order-characterization-1}
produces a derivative constituent for which the exceptional
pole order is exactly $d$. Hence equality holds.
\end{proof}

\begin{proof}[Proof of Theorem~\ref{1.2}]
Put
\[
H(s)=\underset{\substack{0\leq q<n\\i}}{\operatorname{l.c.m.}}
\left\{L_{\mathrm{ex}}(s,\pi_i^{(q)},\operatorname{As})^{-1}\right\}.
\]
By Propositions~\ref{Lex-pole-orders} and
\ref{prop:lcm-characterization}, and
Corollary~\ref{max-exceptional-order},
\[
\begin{aligned}
\operatorname{ord}_{s_0}H
&=\max_{q,i}d_{\mathrm{ex}}(s_0;\pi_i^{(q)})\\
&=\operatorname{pord}_{s_0}L(s,\pi,\operatorname{As})\\
&=\operatorname{ord}_{s_0}L(s,\pi,\operatorname{As})^{-1}
\end{aligned}
\]
for every $s_0\in\mathbb C$. Both functions are finite
products of inverse $\Gamma_{\mathbb R}$-factors. The uniqueness
assertion in Proposition~\ref{prop:lcm-characterization}
therefore gives $H(s)=L(s,\pi,\operatorname{As})^{-1}$.
\end{proof}



\begin{funding}
The author is supported by the Global-Learning \& Academic research
institution for Master’s·PhD students, and Postdocs (G-LAMP) Program
of the National Research Foundation of Korea (NRF), funded by the
Ministry of Education (No.~RS-2025-25442252).
\end{funding}

\end{section}

\end{document}